\documentclass[11pt]{amsart}

\usepackage{fullpage}
\usepackage{setspace}
\usepackage[leqno]{amsmath}
\usepackage{amsfonts}
\usepackage{amssymb}
\usepackage{amsthm}
\usepackage{amssymb}
\usepackage{tikz}
\usetikzlibrary{fpu}
\usepackage{fp}
\usepackage{parskip}
\usepackage{xfrac}
\usepackage{bbm}

\theoremstyle{plain}
\newtheorem{theorem}{Theorem}[section]
\newtheorem{corollary}[theorem]{Corollary}
\newtheorem{lemma}[theorem]{Lemma}

\newtheorem{proposition}[theorem]{Proposition}

\theoremstyle{definition}

\def\FF{\mathbb F}

\def\cc{\mathcal C}

\def\ff{\mathcal F}
\def\ll{\mathcal L}
\def\stb{\big |}
\def\vs{\vspace{0.05in}}

\title{Companion Lie Algebras of Leibniz Algebras}

\author[McAlister]{Allison McAlister}
\address{Department of Mathematics and Computer Science, High Point University\\
High Point, NC 27268}
\email{amcalist@highpoint.edu}

\author[Stitzinger]{Ernie Stitzinger}
\address{Department of Mathematics, North Carolina State University\\
Raleigh, NC 27695}
\email{stitz@ncsu.edu}

\author[White]{Ashley White}
\address{Department of Mathematics, North Carolina State University\\
Raleigh, NC 27695}
\email{anwalls@ncsu.edu}

\date{\today}

\begin{document}

\begin{abstract}
In extending results from Lie to Leibniz algebras, it is helpful to have techniques which translate results from the former to the latter without having to repeat the (perhaps modified) arguments. Such a technique is developed in this work, and some applications are given.


\vs


\end{abstract}

\maketitle

\section{Introduction}\label{intro}

A Leibniz algebra, L, is an algebra with multiplication defined by the (left) Leibniz identity
$$ a(bc) = (ab)c + b(ac). $$ 
This multiplication need not be antisymmetric. 
If this multiplication is, in fact, antisymmetric, then L is a Lie algebra. 
Leibniz algebras have been studied at length in \cite{loday}, \cite{barnes}, \cite{ayupovomirov}, \cite{summary}, and other works.  
In this paper, we consider L to be a finite-dimensional Leibniz algebra over an arbitrary field unless otherwise stated.      

\vs

The \textit{Frattini subalgebra} of L, denoted F(L), is the intersection of all maximal subalgebras of L. 
Note that F(L) need not be an ideal of L. 
Thus, we consider the \textit{Frattini ideal} of L, denoted $\Phi(L)$, which is the largest ideal of L contained in the Frattini subalgebra of L. 
These structures have been studied at length in \cite{fratleib}.  
It is often especially useful to consider a Leibniz algebra whose Frattini ideal is 0. 
A particularly well-known Lie algebra result that carries over to the Leibniz case is the following: The Frattini ideal of L/$\Phi$(L) is 0 for any Leibniz algebra L.

\section{Companion Lie Algebra}

Let $\ff$ be a subalgebra closed formation of solvable Leibniz algebras (see \cite{barnes} and \cite{summary} for an introduction to Leibniz algebras and \cite{barnessolv} for results on formations). 
Examples of such formations are the classes of nilpotent, supersolvable, and strongly solvable Leibniz algebras. 
Let L be a solvable Leibniz algebra that is minimally not in $\ff$ with $\phi$(L) = 0. 
If A is a minimal ideal in L, then A is complemented in L by a subalgebra. 
We claim that there is a unique minimal ideal of L. 
Suppose that A and B are minimal ideals of L. 
They are complemented by subalgebras H and K, respectively. 
H and K are both in $\ff$ by the minimality of L. 
Then L = L/(A $\cap$ B), which is in $\ff$ since $\ff$ is a formation. 
This contradicts the assumption that L does not belong to $\ff$.
Hence, L has a unique minimal ideal, A. 
Thus, Soc(L) = A.
Suppose that L is not Lie.

\vs

Let N(L) denote the nilradical of L. 
Since $\phi$(L) = 0, N(L) = A = Z$_L$(A) $\subset$ Leib(L) \cite{fratleib}. 
Furthermore, consider the mapping $\ll$: x $\mapsto$ L$_x \stb_A$. 
$\ll$ is a homomorphism from L into the derivation algebra of A. 
L is the right centralizer of A (since A $\subset$ Leib(L)), and A = Z$_L$(A).  
Hence, A is the kernel of $\ll$. 
Hence, H = L/A is Lie, and A = Leib(L). 

\vs

We construct a companion Lie algebra, C, for this Leibniz algebra, L = A + H, where A is the unique minimal ideal in L, and H is the subalgebra of L complementing A. 
Let C be the vector space direct sum of A and H with the product [a+h, b+k] = [h,b] - [k,a] + [h,k], where the products on the right-hand side of this equation are the same as in L. 
It is verified that C is a Lie algebra. 
Note that such correspondences have been considered previously \cite{Leibnizcharp}. 
There the reverse construction is developed, obtaining a Leibniz algebra from a Lie algebra. 
Their conditions and purposes are different, but the construction is the inverse of ours.

\vs

\vs

\vs

\begin{lemma}
\label{companionlie}
Let $\ff$ be a formation of solvable Leibniz algebras. Let L be minimally not in $\ff$, and suppose that $\phi$(L) = 0. Then L = A + H, where A is the unique minimal ideal in L. 

If L is not a Lie algebra, then A = Leib(L) is self-centralizing in L, and H is a Lie algebra. Then the algebra C constructed as C = A + H with product [a+h, b+k] = [h, b] - [k, a] + [h, k], where the right-hand-side products are the same as in L, is a Lie algebra. We call C the companion Lie algebra of L.
\end{lemma}

\section{Strongly Solvable Leibniz Algebras}

As in the case of Lie algebras, a Leibniz algebra L is called \textit{strongly solvable} if L$^2$ is nilpotent. 
Suppose $\ff$ is the formation of strongly solvable Leibniz algebras with L, C, A, and H as in the previous section. 
Let x $\in$ L, and let $x$ also denote the corresponding element in C. 
Let $L_x$ denote left multiplication of L by x, and let ad $x$ denote left multiplication of $\cc$ by $x$ under the above bracket. 
If x is in A, then $L_x$ and ad x are both nilpotent, although perhaps not equal. 
Suppose that L is a $\phi$-free, solvable, minimally non strongly solvable Leibniz algebra. 
We claim that C also has these properties.

\vs

If x is in H$^2$, then L$_x$ and ad $x$ are essentially the same, although acting on L and C, respectively, and are either both nilpotent or both non-nilpotent. 
If all x $\in$ H$^2$ have ad $x$ nilpotent, then left multiplication by each element in the Lie set H$^2$ $\cup$ A is nilpotent on both L$^2$ and C$^2$. 
Thus, L$^2$ is nilpotent if and only if C$^2$ is nilpotent; hence, C is not strongly solvable.

\vs

A is self-centralizing in L, and hence, A is also self-centralizing in C. 
Thus, A is the unique minimal ideal in C.  
Then H is a maximal subalgebra in C and thus, contains $\phi$(C). 
We claim that $\phi$(C) = 0. 
Suppose not. 
H is a maximal subalgebra of C, and hence, $\phi$(C) $\subset$ H. 
A is the unique minimal ideal in C, and hence, A is contained in $\phi$(C), which is a contradiction. 
Thus, the claim holds.

\vs

We also claim that any proper subalgebra, K, of C is strongly solvable. 
Suppose that C = A + K. 
Then K $\cap$ A is equal to either A or 0 since A is a minimal ideal of C. 

\vs

In the former case, A is contained in K, and then K is equal to C, which contradicts K being a proper subalgebra of C.  
In the latter case, K is isomorphic to H, which is strongly solvable in L, and thus, in C. 
Thus, K is strongly solvable.

\vs

Now suppose that C $\neq $ A + K. 
We will show that A + K is strongly solvable, and hence, that K is also strongly solvable. 
Thus, we assume that A $\subseteq$ K. 

\vs

Then K = K $\cap$ (H + A) = (K $\cap$ H) + A, which is a subalgebra of L by the correspondence between L and C. 
Then K$^2$ = (K $\cap$ H)$^2$ + A in Leibniz algebra L, and analogously, K$^2$ = (K $\cap$ H)$^2$ + A  in Lie algebra C. 
In L, both summands of K$^2$ are nilpotent Lie sets, and likewise, the corresponding summands are nilpotent in C. 
Thus, K$^2$ is nilpotent as a subalgebra of C.  
Hence, K is strongly solvable for K any proper subalgebra of C. 
Since C is itself not strongly solvable, C is a solvable, minimally non-strongly solvable Lie algebra. 

\vs

Such Lie algebras have been described in Theorem 3.1 of \cite{btv}.

\vs

\vs

\vs

\begin{theorem}
\label{minimalnonstronglysolvablethm}
Let L be a solvable $\phi$-free minimally non-strongly solvable Leibniz algebra over field $\FF$. 
Then $\FF$ has characteristic p $>$ 0, and L = A + H is a semidirect sum, where
	\begin{enumerate}
	\item A is the unique minimal ideal of L, 
	\item $\dim$ A $\geq$ 2, 
	\item A$^2$ = 0, and 
	\item either H = M $\dotplus \; \FF$x, where M is a minimal abelian ideal of H, or H is the three-dimensional Heisenberg algebra.	
	\end{enumerate}

Either L is Lie, or if not, then H is Lie, A = Leib(L), and AL = 0.

\end{theorem}

\begin{proof}
If L is a Lie algebra, then the result is Theorem 3.1 of \cite{btv}. 
If L is not Lie, then the companion Lie algebra, C, satisfies the conditions and conclusions of this theorem. 
Thus, L does also by Lemma \ref{companionlie}
\end{proof}

In \cite{btv}, Bowman, Towers, and Varea prove the Lie algebra version of the next result from Theorem \ref{minimalnonstronglysolvablethm}. 
We obtain it from \cite{tworecogleib} instead. 

\vs

\vs

\begin{proposition}
\label{minnonstrsolv2genprop}
Let L be a solvable Leibniz algebra that is minimally non-strongly solvable. Then L is two-generated. 
\end{proposition}

\begin{proof}
If L is not two-generated, then all two-generated subalgebras are strongly solvable. 
Then the result follows from Corollary 1 of Theorem 2 of \cite{tworecogleib}. 
\end{proof}

\vs

\begin{proposition}
\label{propersubalg}
Let L be a solvable Leibniz algebra with each two-generated proper subalgebra strongly solvable. Then
	\begin{enumerate}
	\item L is either strongly solvable or two-generated, and
	\item Every proper subalgebra of L is strongly solvable.
	\end{enumerate}
\end{proposition}

\begin{proof}
The proof of (1) follows the proof of \ref{minnonstrsolv2genprop}. 

\vs

We now prove (2). Let S be a minimally non-strongly solvable subalgebra of L. 
If S is two-generated, then S is strongly solvable by hypothesis. 
Hence, S is not two-generated. 
Thus, all two-generated subalgebras of S are strongly solvable, and S is strongly solvable also.
\end{proof}

\section{Supersolvable Leibniz Algebras}

A Leibniz algebra L is called \emph{supersolvable} if there exists a chain of ideals 0 = L$_0 \subset$ L$_1 \subset$ L$_2 \subset \hdots$ L$_{n-1} \subset$ L$_n$ = L, where L$_i$ is an $i$-dimensional ideal of L. 
In this section, supersolvability will be considered in the same manner as was strong solvability in the previous section. 
Let L be a solvable, $\phi$-free,minimally non-supersolvable Leibniz algebra. 
Supersolvable Leibniz algebras form a formation; hence, Lemma \ref{companionlie} applies. 
Then L = A + H as in Lemma \ref{companionlie}. 
Now, if L is Lie, then the structure of L has been determined in \cite{elduquevarea} and set forth in \cite{btv}. 
We state the result as 

\vs

\vs

\begin{theorem}
\label{btvsupersolvablethm}
Let L be a solvable, minimally non-supersolvable Lie algebra which is $\phi$-free. Then the candidates for L are:
	\begin{enumerate}
	\item If L is strongly solvable, L = A $\dotplus \langle x \rangle$, where A is the unique minimal ideal of L and $\dim$ A $>$ 1.
	\item If L is not strongly solvable, then F is of characteristic p $>$ 0, L has unique minimal ideal A with basis $\{e_1, \hdots, e_p\}$, and one of the following holds:
		\begin{enumerate}
		\item L = A + $\langle x, y \rangle$ with antisymmetric multiplication $xe_i = e_{i+1}$, \\ $ye_i = (\alpha + i) e_i$, with indices mod p, yx = x, and for all $a \in$ F, $a = t^p - t$ for some $t \in$ F, or
		\item L = A + $\langle x, y, z \rangle$ with anti-symmetric multiplication xe$_i$ = e$_{i+1}$, \\ ye$_i$ = (i + 1) e$_{i-1}$, ze$_i$ = e$_i$, with indices mod p, yx = z, xz = yz = 0, and F is perfect when p = 2.
		\end{enumerate}
	\end{enumerate}
\end{theorem}

The list of algebras in Theorem \ref{btvsupersolvablethm} are the possible minimally non-supersolvable Lie algebras, however, it is important to note that not all algebras of these forms have this property. 
However, if F is algebraically closed, the algebras in (2) are minimally non-supersolvable, while the algebra in (1) must be supersolvable. 
Hence, all minimally non-supersolvable Lie algebras over an algebraically closed field are as in (2).   

\vs

Now suppose that L is not a Lie algebra. 
Then L is as in Lemma \ref{companionlie}; 
L = A $\dotplus$ H, where A = Leib(L), and H is as in the lemma. 
Take C to be the companion Lie algebra defined in Lemma \ref{companionlie}. 
A is the minimal ideal in C, and C is $\phi$-free.

\vs

In Leibniz algebra L, H is supersolvable by assumption, and A is abelian. 
However, A is irreducible under the action of H acting on the left. 
Since L is not supersolvable by hypothesis, the dimension of A is greater than 1. 
The left action of H on A is the same in C as in L, and hence, $\dim$ A is greater than one and C is not supersolvable.

\vs

We show that all proper subalgebras of C are supersolvable, and hence, that C is a solvable, minimally non-supersolvable Lie algebra. 
Let K be a proper subalgebra of C. 
If C = A + K, then either A $\cap$ K = 0 or A $\cap$ K = A since A is also a minimal ideal in C. 
In the former case, K is isomorphic to H and hence, is supersolvable. 
In the latter case, K = C, which is a contradiction. 

\vs

Now suppose that A + K is a proper subalgebra of C. 
We show that A + K is supersolvable, and hence, that K is also supersolvable. 
We may assume that A $\subset$ K. 
Then as above, K = (K $\cap$ H) + A. 
Note that K $\cap$ H is the same in L and C. 

\vs

In L, the action of K $\cap$ H on A is simultaneously triangularizable. 
The left action of K $\cap$ H on A in C is the same as in L, and hence, K is supersolvable in Lie algebra C. 
Then C is minimal non-supersolvable and is as in Theorem \ref{btvsupersolvablethm}. 
Hence, we have the following theorem.

\vs

\vs

\begin{theorem}
\label{leibnizsupersolvthm}
Let L be a solvable, $\phi$-free, minimally non-supersolvable Leibniz algebra. 
If L is a Lie algebra, then L is as in Theorem \ref{btvsupersolvablethm}. If L is not Lie, then L is as in Theorem \ref{btvsupersolvablethm} with the added conditions that A = Leib(L) and AL = 0.
\end{theorem}

\vs

\vs

\begin{corollary}
If L is as in Theorem \ref{leibnizsupersolvthm} and F is algebraically closed, then the solvable, minimally non-supersolvable Leibniz algebras are precisely the algebras at characteristic p such that either

\begin{enumerate}
\item[(a.)] L = A + B where A = $(( e_1, \hdots, e_p ))$ is the unique minimal ideal which is abelian, B = $\langle x, y \rangle$ with $y e_i$ = $- ie_i$ and $x e_i$ = $- e_{i+1}$ (indices mod p), $yx$ = x, and either A = Leib(L) or L is Lie, or

\item[(b.)] L = A + B where A = $(( e_1, \hdots e_p ))$ is the unique minimal ideal which is abelian, B is the Heisenberg Lie algebra with basis $\{x, y, z\}$ and multiplication $y e_i$ = (i+1)$e_{i-1}$, $x e_i$ = $- e_{i+1}$, $z e_i$ = $e_i$, yx = z, and either A = Leib(L) or L is Lie. Again, the index action is mod p.
 
\end{enumerate}
\end{corollary}

Several results similar to Theorem 3.5 in \cite{btv} are now obtained in the Leibniz algebra case. 
Note that if Leibniz algebra L is solvable and minimally non-supersolvable, then all two-generated proper subalgebras are supersolvable. 
If L is not two-generated, then all two-generated subalgebras are supersolvable, and hence L is supersolvable from Theorem 4 of \cite{tworecogleib}, a contradiction. Thus, we have the following.

\vs

\vs

\begin{theorem}
If L is solvable and minimal non-supersolvable, then L is two-generated.
\end{theorem}

\vs

\vs

\begin{theorem}
If every two-generated proper subalgebra of L is supersolvable, L itself is not supersolvable, and L is $\phi$-free, then L has the structure of one of the algebras in Theorem \ref{leibnizsupersolvthm}.
\end{theorem}

\begin{proof}
Let S be a proper subalgebra of L. 
If S is two-generated, then it is supersolvable by assumption. 
If S is non two-generated, then every two-generated subalgebra of S is supersolvable by assumption. 
Since supersolvability is two-recognizeable \cite{tworecogleib}, S is also supersolvable. 
Thus, all proper subalgebras of L are supersolvable, and L is minimally non-supersolvable. 
Hence, Theorem \ref{leibnizsupersolvthm} applies. 
\end{proof}

\section{Triangulable Leibniz Algebras}

A Leibniz algebra L over a field F is triangulable on L-module M if when F is extended to K, the algebraic closure of F, K $\otimes$ M admits a basis such that the representing matrices of K $\otimes$ L are upper triangular. 
L is said to be nil on M if left multiplication on M by each x $\in$ L is nilpotent. 
Then L acts nilpotently on M \cite{jacobsonleib}. 
There is a maximal ideal of L that acts nilpotently on M.
This ideal is denoted by nil(L). 
If L$^2$ is contained in nil(L), then the same holds in the algebra and module over K. 
By Theorem 1 of \cite{triangulable}, in the algebraically closed case, L is triangulable on M if L$^2$ is contained in nil(L). 
Hence, in the general case, L is triangulable on M if L$^2$ is in nil(L). 
Note that these definitions still apply in the special case that L is a subalgebra of M. 
We state them in this context.

\vs

\vs

\begin{proposition}
\label{triangulabledefns}
A subalgebra L of Leibniz algebra M is triangulable on M if and only if L$^2$ is contained in nil(L). 
If L is an ideal of M, then L is triangulable on M if and only if L$^2$ is nilpotent. 
M is triangulable on itself if and only if M is strongly solvable.
\end{proposition}

The following two propositions follow exactly as in their Lie algebra versions, with left multiplication replacing ad in the proofs as shown in Lemma 4.1 and Lemma 4.5 of \cite{btv}.

\vs

\vs

\begin{proposition}
Let L be a Leibniz algebra, and let S and T be subalgebras of L that are nil on S such that S is contained in the normalizer of T. 
Then S + T is nil on L. 
\end{proposition}

\vs

\vs

\begin{proposition}
If S is a subalgebra of L such that $\phi$(L) is contained in S, then nil (S/$\phi$(L)) = nil(S)/$\phi$(L).
\end{proposition}

\begin{proof}
Clearly $\phi$(L) is contained in nil(L) since $\phi$(L) is a nilpotent ideal in L. 
Let nil(S/$\phi$(L)) = J/$\phi$(L). 
Clearly nil(S) is contained in J. 
Let x $\in$ J. 
Now L$_x$ acts nilpotently on L/$\phi$(L), and L = $\phi$(L) + L$_0$(x), the Fitting null component of L$_x$ acting on L. 
Since L$_0$(x) is a subalgebra of L supplementing the Frattini ideal, L$_0$(x) is equal to L. 
This holds for all such x, and hence, J is contained in nil(S), and the result holds.
\end{proof}

It is known that a solvable Leibniz algebra is triangulable on itself if all two-generated subalgebras of L are triangulable \cite{triangulable}. 
In \cite{btv}, Lie algebras all of whose \emph{proper} two-generated subalgebras are triangulable  are similarly investigated. 
Our purpose is to find Leibniz algebra analogues to this and related ideas. 

\vs

\vs

\begin{theorem}
\label{2gentriangulable}
Let L be a solvable Leibniz algebra such that each two-generated proper subalgebra is triangulable on L. 
Then L is triangulable. 
\end{theorem}

\begin{proof}
Each two-generated proper subalgebra of L is strongly solvable. 
Then by Proposition \ref{propersubalg}, every proper subalgebra of L is strongly solvable, as is each proper subalgebra of L* = L/$\phi$(L).

\vs

If L is not triangulable, then L is not strongly solvable, and neither is L*. 
Thus, Theorem \ref{minimalnonstronglysolvablethm} applies, L$^*$ = A + B as in the theorem, and (L*)$^2$ = A + B$^2$. 
B is two-generated. 
Hence B is triangulable on L*. 
Thus, B$^2$ acts nilpotently on L* and hence, on A. 
Thus, (L*)$^2$ is nilpotent and L* is strongly solvable, a contradiction.
\end{proof}

It is interesting to note that L is triangulable if and only if it is strongly solvable. 
If all two-generated subalgebras are triangulable on L, then L is triangulable (Theorem \ref{2gentriangulable}). 
However, if all two-generated subalgebras are strongly solvable, this does not guarantee that L is triangulable since a subalgebra can be strongly solvable without being triangulable on the algebra. 
The algebras in Theorem \ref{minimalnonstronglysolvablethm} are of this type. 

\vs

\vs

\begin{theorem}
If L is minimally non-triangulable (every proper subalgebra of L is triangulable on L, but L itself is not), then L is two-generated and L/$\phi$(L) is simple.
\end{theorem}

\begin{proof}
By Theorem \ref{2gentriangulable}, L is not solvable. 
If L* = L/$\phi$(L) is triangulable, then (L*)$^2$ is nilpotent and L* is solvable, a contradiction. 
Thereofre, L* is not triangulable, and hence not strongly solvable.

\vs

Consider a proper subalgebra S of L and the corresponding proper subalgebra S* = S/$\phi$(L) of L*. 
S is triangulable on L, hence S$^2$ is contained in nil(S) and (S*)$^2$ is contained in nil(S*). 
Hence, S* is strongly solvable. 

\vs 

Note that the nilradical of L* is not 0. 
Since $\phi$(L*) = 0, nil(L*) is complemented in L* by a subalgebra T \cite{fratleib}, which is strongly solvable. 
Hence, L* = nil(L*) + T is solvable, a contradiction. 
Thus, L* is semisimple. 
If L* were to contain a proper ideal, that ideal would be strongly solvable and hence solvable, a contradiction. 
Thus, L* contains no proper ideals, and L* is simple. 

\vs

\end{proof}


\bibliography{leibnizbib}
\bibliographystyle{plain}

\end{document}